\documentclass{amsart}
\usepackage[utf8]{inputenc}
\usepackage{amsmath}
\usepackage{amssymb}
\usepackage{verbatim}
\usepackage{graphicx}

\usepackage{tikz-cd}
\usepackage{geometry}
\usepackage{esint}
\usepackage{soul}
\usepackage[colorlinks=true,urlcolor=blue,citecolor=magenta,linkcolor=cyan, menucolor=teal,linktocpage,pdfpagelabels,bookmarksnumbered,bookmarksopen]{hyperref}
\usepackage[hyperpageref]{backref}

\numberwithin{equation}{section}


\title{Large time behavior of fractional porous media equation}

\author[G. Franzina]{Giovanni Franzina}
\address[G. Franzina]{Istituto per le Applicazioni del Calcolo ``M. Picone''
\newline\indent
Consiglio Nazionale delle Ricerche
\newline\indent 
Via dei Taurini, 19, 00185 Roma, Italy}
\email{giovanni.franzina@cnr.it}

\author[B. Volzone]{Bruno Volzone}
\address[B. Volzone]{ Dipartimento di Scienze Economiche, Giuridiche, Informatiche e Motorie - DiSEGIM,
\newline\indent
Universit\`a degli Studi di Napoli ``Parthenope''
\newline\indent
Via Guglielmo Pepe Rione Gescal - 80035 Nola (NA), Italy
}
\email{bruno.volzone@uniparthenope.it}

\newcommand{\inter}[1]{\ensuremath{\frac{(#1(x)-#1(y))^2}{|x-y|^{N+2s}}}}
\newcommand{\interr}[2]{\ensuremath{\frac{(#1(x)-#1(y))(#2(x)-#2(y))}{|x-y|^{N+2s}}}}
\newcommand{\sob}{{\mathcal D}^{s,2}_0(\Omega)}
\font\script=rsfs10 at 12pt
\def\F{{\mbox{\script F}\,\,}}
\newcommand{\lyap}[1]{\ensuremath{\F_{\!\!\!\!q,\alpha}^{\!s}}\!\left(#1\right)}
\newcommand{\lyapm}[2]{\ensuremath{\F_{\!\!\!\!#1,\alpha}^{\!s}}\!\left(#2\right)}
\newcommand{\lyapno}{{\F_{\!\!\!\!q,\alpha}^{\!s}}}
\newcommand{\RN}{\mathbb R^N}
\newcommand{\R}{\mathbb R}

\numberwithin{equation}{section}

\newtheorem{theorem}{Theorem}[section]
\newtheorem{proposition}{Proposition}[section]
\newtheorem{lemma}{Lemma}[section]
\newtheorem{corollary}{Corollary}[section]

\theoremstyle{definition}
\newtheorem{definition}{Definition}[section]
\newtheorem{remark}{Remark}[section]
\newtheorem*{ack}{Acknowledgments}

\begin{document}

\begin{abstract}
Following the methodology of \cite{bravol}, we study the long-time behavior for the signed Fractional Porous Medium Equation in open bounded sets with smooth boundary. Homogeneous exterior Dirichlet boundary conditions are considered. We prove that if the initial datum has sufficiently small energy, then the solution, once suitably rescaled, converges to a nontrivial constant sign solution of a sublinear fractional Lane-Emden equation.
\par
Furthermore, we give a nonlocal sufficient energetic criterion on the initial datum, which is important to identify the exact limit profile, namely the positive solution or the negative one.
\end{abstract}
\maketitle

\tableofcontents

\section{Introduction}
In this paper, we will achieve some stabilization results for solutions to an initial boundary value problem for the {\it Fractional Porous Medium Equation} (FPME for short), of the form
\begin{equation}\label{FPMEintro}
\left\{\begin{array}{rcll}
\partial_{t}u&=&-(-\Delta)^{s} (|u|^{m-1}u), & \mbox{ in } Q:=\Omega\times (0,+\infty),\\
u&=&0, & \mbox{ in } \R^{N}\setminus\Omega\times(0,+\infty),\\
u(\cdot,0)&=&u_0,& \mbox{ in } \Omega.
\end{array}
\right.
\end{equation}
Here we consider the porous medium regime, \emph{i.e.} $m>1$, we assume $0<s<1$, and $\Omega$ is a bounded open set of $\mathbb R^N$.
 A broad theory has been developed for this problem under several aspects (existence, uniqueness, regularity etc.), see for instance \cite{BFRO,BFV,BSV,BV,BV2}. The main result of this paper concerns 
solutions emanating from initial data $u_0$ for which the energy functional
\[
\frac12\int_{\RN}\int_{\RN} \inter{\varphi}\,dx\,dy-\frac{m+1}{(m-1)m}
\int_\Omega |\varphi|^{\frac{m+1}{m}}\,dx
\]
does not exceed its first excited level when the choice $\varphi= |u_0|^{m-1}u_0$ is made. By following the methodology of~\cite{bravol}, in which
the local case was considered, we compute the large time asymptotic profile of such solutions in this non-local framework. As in the local case, sign-changing initial data are included in the analysis: irrespective of their sign, if their energy is small enough then they give rise to solutions that in the large time limit are asymptotic to functions with a spatial profile arising in the energy minimization.
\medskip

For a precise statement, we need to introduce, for all $q\in (1,2)$ and $\alpha\in(0,+\infty)$, the functional
defined on $\sob$ (see Sect.~\ref{sec2} for the precise definition of this space) by
\begin{equation}
\label{functional}
\ensuremath{\F_{\!\!\!\!q,\alpha}^{\!s}}\!\left(\varphi \right) = 
\frac12\int_{\RN}\int_{\RN} \inter{\varphi}\,dx\,dy-\frac{\alpha}{q}
\int_\Omega |\varphi|^{q}\,dx\,,
\end{equation}
whose critical points, 
by definition, are the weak solutions of the Lane-Emden equation
\begin{equation}
\label{ELLintro}
	(-\Delta)^s\varphi = \alpha\ |\varphi|^{q-2}\varphi\,,\qquad\text{in $\Omega$,}
\end{equation}
with homogeneous Dirichlet boundary conditions. It is known~\cite{fralic} that the minimal energy 
\begin{equation}
\label{minimizer}
\Lambda_1=\min\{\lyap\varphi\mathbin{\colon}\varphi\in \sob\}\,,
\end{equation}
is achieved by a solution with constant sign, that it is unique (up to the sign). Also, we set
\[
	\Phi(u)= |u|^{m-1}u
\]
and we observe that $\Phi^{-1}(\varphi) = |\varphi|^{q-2}\varphi$ where $q=(m+1)/m$. Now, following \cite{bravol}, we define the \emph{second critical energy level}, or {\em first excited level}, as it follows
\[
\Lambda_{2}=\inf\Big\{\Lambda>\Lambda_1 \mathbin{\colon}\text{$\Lambda$ is a critical value of $\lyapno$}\Big\}.
\]
The the solution to problem  \eqref{FPMEintro} has the following stabilization property.
\begin{theorem}\label{mainthm1}
Let $m>1$, $0<s<1$, and let $\Omega$ be a bounded open set in $\mathbb R^N$ with $C^{1,1}$ boundary. Given
$u_0\in L^{m+1}(\Omega)$, with $\Phi(u_0)\in \sob$ and 
\[
\ensuremath{\F_{\!\!\!\!q,\alpha}^{\!s}}\!\left(\Phi(u_0)\right)
	<\Lambda_{2} \,,\ \text{where $q=\frac{m+1}{m}$ and $\alpha=\frac{1}{m-1}$,}
\]
let $u$ be the weak solution of the fractional porous media equation \eqref{FPMEintro} with initial datum $u_0$. Then,
\[
	\lim_{t\to\infty} \| t^\alpha u(\cdot,t)-|W|^{q-2}W\|_{L^{m+1}(\Omega)}=0\,,
\]
where $W\in\{w_\Omega,-w_\Omega\}$ and $w_\Omega$ is the positive minimiser of $\lyapno$
on $\sob$.
\end{theorem}

Besides on $\Omega$, the function $w_\Omega$ that
achieves the minimum in \eqref{minimizer} depends on $m$ (through $\alpha$ and $q$) and on $s$; we refer to the material in Sect.~\ref{sec:ell} for
its existence and uniqueness, that are however well known.
Recall that, in the case of nonnegative data $u_{0}$, it is also well known (see \cite{BFV,BSV}) that the solution $u$ stabilizes towards the so called $\emph{Friendly Giant}$, following the denomination due to Dahlberg and Kenig for the standard porous medium equation \cite{DK} (see also \cite[Sec.\ 5.9]{VaBook}), namely
\[
S(x,t)=t^{-\alpha}w_{\Omega}(x)^{q-1}.
\]
That is a separate variable solution taking $+\infty$ as initial value.
In particular, in~\cite{BFV,BSV} various interesting results are shown, related to the finer problem of the sharp convergence rate of the relative error, a question that was also faced in the classical paper \cite{AP}.
\medskip

As said,  
Theorem~\ref{mainthm1} can be proved via the approach used in~\cite{bravol} to deal with the local problem, thanks to a Lyapunov-type
property of the energy functional \eqref{functional}: namely, that
\[
	t\longmapsto \lyap{\Phi(v(\cdot,t))} \text{is non-increasing,}
\] 
whenever 
$v$ is an energy solution (see Sect.~\ref{sec:4} for precise definitions) of the initial boundary value problem for the rescaled equation
\begin{equation}
\label{FPMErescintro}
	\partial_t v + (-\Delta )^s\Phi(v) = \alpha v\,.
\end{equation}
This property is inferred, in this paper, from 
an entropy-entropy dissipation estimate in Sect.~\ref{sec:4}.
In order to prove it, we produce solutions by the classical Euler implicit time discretization scheme.
That has the advantage of providing a discrete version of the desired inequality in which we can pass to the limit. We prefer this approach to considering exact solutions to a uniformly parabolic approximation, as done in~\cite{bravol} for the local problem, mainly because
that would require $C^{1,\alpha}$ estimates for the non-local operators that are 
obtained by regularizing the {\em signed} FPME; incidentally, we mention that strong results of this type can be found in~\cite{BFRO} in the case of {\em non-negative} solutions.
\medskip

We recall here in brief the use of the the Lyapunov property for the proof of Theorem~\ref{mainthm1}. Given a solution $u$ of \eqref{FPMEintro},
the equation \eqref{FPMErescintro}  for the function $v(x,t)=e^{\alpha t} u(x,e^t-1)$ describes a system that evolves, irrespective of the starting conditions, to fixed points, {\em i.e.}, states of the form $\Phi^{-1}(v)$ with $v$ being a critical point of the energy functional. Because of the isolation of the energy minimizing solutions $\pm w_\Omega$,
proved in~\cite{fralic}, this and the Lyapunov property imply that
the disconnected set $\{\pm\Phi^{-1}(w_\Omega)\}$
has a non-empty basin of attraction, including all initial states with energy smaller than the first excited level.
(The isolation property implies some restriction on boundary regularity: assumptions weaker than those made in the our main statement are also feasible, but that is not the object of this paper.)

Eventually, by the compactness of the relevant Sobolev embedding,
by energy coercivity, and by the Lyapunov property, orbits are relatively compact; thus, the $\omega$-limit
is connected and
the only possible cluster point of the orbit emanating from an initial state $u_0$ 
below the energy threshold is either $\Phi^{-1}(w_\Omega)=w_\Omega^{q-1}$ or $\Phi^{-1}(-w_\Omega)=-w_\Omega^{q-1}$.
\medskip

Yet, meeting the threshold requirement in Theorem~\ref{mainthm1} implies no restriction on the sign that $u_0$ should take in $\Omega$,
nonetheless. The relevance in energy of the nodal sets, instead, enters in predicting which one of the two possible limit profiles the orbit will accumulate
to. The following Proposition, which is the non-local counterpart of an analogous result of~\cite{bravol},
quantifies this idea.

\begin{proposition}\label{prop:select}
In  the assumption of Theorem~\ref{mainthm1}, we have $W=w$ if either
\begin{subequations}\label{assumnonloc}
\begin{equation}
\label{assnl1}
\text{ $\lyap{\Phi(u_0^-)}>0$ and $\lyap{\Phi(u_0)}<\Lambda_2$}
\end{equation}
or  
\begin{equation}\label{assnl2}
\text{$\lyap{\Phi(u_0^-)}\le0$ and  $\displaystyle\lyap{\Phi(u_0^+)}+2\iint\frac{\Phi(u_0)^+(x)\Phi(u_0)^-(y)}{|x-y|^{N+2s}}\,dx\,dy<\Lambda_2$.}
\end{equation}
\end{subequations}
\end{proposition}

We observe that Assumption\ \eqref{assnl2} is consistent with the analogous one
made in the local case in~\cite[Proposition 1.4]{bravol}. In that respect, note that
the double integral disappears in the limit
as $s\to1$,  if renormalized by a degenerating factor (we refer to Remark~\ref{lims1} below for more details).
\medskip

The proof of Proposition~\ref{prop:select} is by contradiction and makes use of a {\em hidden convexity} of Gagliardo's seminorm. This property, 
under the assumptions \eqref{assumnonloc}, allows one to ``prolong in the past'' the orbits $v(\cdot,t)$
of \eqref{FPMErescintro} that stabilize towards $-w_\Omega^{q-1}$ by a trajectory defined for negative times, connecting the initial datum $u_0$
to $w_\Omega^{q-1}$, with an energy control. It turns out, also in view again of the Lyapunov property, that this would contradict the mountain pass-type description 
\[
	\inf_{\gamma} \max_{\varphi\in {\rm Im}(\gamma)} \lyap{\varphi}
\]
of an excited level. In order to see this, in Sect.~\ref{sec:ell} we formulate this variational principle in a way that differs from standards in that the admissible $\gamma$, joining $w_\Omega$ and $-w_\Omega$, are only required to be continuous
with values in the class of real valued measurable functions on $\Omega$ endowed with the topology of the convergence in measure,
rather than the strong topology of Sobolev spaces.

\begin{ack}
B.\ Volzone was partially supported by {\scshape gnampa} of the Italian {\sc in}d{\sc am} (National Institute of High Mathematics).
Part of this material
is based upon work supported by the Swedish Research Council under grant no. 2016-06596 while G.\ Franzina was in residence at {\sc Institut Mittag-Leffler} in Djursholm, Sweden during the Research Program {\em “Geometric Aspects of Nonlinear Partial Differential Equations”} in 2022. 
We thank L.~Brasco for many interesting discussions contributing to improve the quality of the paper, and for suggesting the proof in appendix.
B.~Volzone wishes to thank M.~Bonforte and M.~Muratori for many fruitful discussions. 
\end{ack}

\section{Notations and Assumptions}\label{sec2}
Throughout this paper, 
we assume $\Omega$ to be an open bounded set, we take $s\in(0,1)$, and we let $m>1$. 
Then, we denote by $\sob$ be the completion of $C^\infty_0(\Omega)$
with respect to the norm
\begin{equation}
\label{seminorm}
	[u]_{s} = \left\{
		\int_{\mathbb R^N}\int_{\mathbb R^N}
	\inter{u}\,dx\,dy
	\right\}^\frac12\,.
\end{equation}
\begin{remark}
Since by assumption $\Omega$ supports a Poincar\'e inequality, $\sob$
is isomorphic to the completion
$\mathcal{X}^{s,2}_0(\Omega)$ with respect to the complete norm, often denoted also by ${\widetilde W}^{s,2}_0(\Omega)$ (see~\cite{BCV,BLP}). If in addition $\Omega$ has a Lipschitz boundary, then
both spaces are isometrically equivalent to the closure of $C^\infty_0(\Omega)$ in $W^{s,2}(\mathbb R^N)$,
consisting of all measurable functions vanishing a.e.\ out of $\Omega$ for which \eqref{seminorm} is finite.
Except if $s=\tfrac12$, in fact, $\sob$ coincides with the closure $W^{s,2}_0(\Omega)$
of $C^\infty_0(\Omega)$ in $W^{s,2}(\Omega)$.
 For more details on this functional-analytic setting, see~\cite[Chap.~3]{EE}. 
\end{remark}

For every $u\in C^2(\Omega)$, we take
\begin{equation}
\label{slap}
(-\Delta)^su(x) = \lim_{\varepsilon\to0^+} \int_{\R^N\setminus B_\varepsilon(x)} \frac{u(x)-u(y)}{|x-y|^{N+2s}}\,dy
\end{equation}
as the definition of the $s$-laplacian of $u$ at point $x$. As $s$ is fixed, we are not interested in multiplying the principal value integral by
any renormalization factor.
\medskip

By $L^0(\Omega)$, we shall denote
the space of all (equivalence classes of) real valued measurable functions on $\Omega$, endowed with the topology of the convergence in measure.
\medskip

We also set
\begin{equation}
\label{power}
	\Phi(v) = |v|^{m-1}v\,,\qquad \text{for all $v\in\R$.}
\end{equation}
For all $1<q<2$ let us define
\begin{equation}
\label{lambda1}
\lambda_1(\Omega,q,s) = \inf_{u\in\sob}\left\{ [u]_{s}^2 \mathbin{\colon}
	\int_\Omega |u|^q\,dx=1\right\}\,.
\end{equation}
Given $1<q<2$ and $\alpha>0$, we consider the functional
\begin{equation}
\label{lyap}
\lyap{\varphi} = \frac12\int_{\RN}\int_{\RN} \inter{\varphi}\,dx\,dy-\frac{\alpha}{q}
\int_\Omega |\varphi|^q\,dx\,,
\end{equation}
for all $\varphi\in\sob $.

\section{Elliptic Toolkit}\label{sec:ell}
The critical points of $\lyapno$ are the weak solutions $u\in \sob$ of the Lane-Emden type equation
\begin{equation}
\label{LEeq}
	(-\Delta)^su =\alpha |u|^{q-2}u\,,\qquad \text{in $\Omega$,}
\end{equation}
which means
\begin{equation}
\label{LEweak}
\int_{\RN}\int_{\RN} \interr{u}{\psi}\, dx\,dy
=\alpha\int_\Omega |u|^{q-2}u\psi\,dx\,,\qquad \text{for all
$\psi \in\sob$.}
\end{equation}

\begin{lemma}\label{lm:coerc}
Let $1<q<2$ and $\alpha>0$. Then,
the functional
 $\lyapno$ is coercive on $\sob$, {\em i.e.},
 \begin{equation}
  \label{coercivity}
 	\lyap{u}\ge \tfrac14 [u]_{s}^2 - C\,,
 \end{equation}
for all $u\in \sob$, where $C$ is a constant depending on $\Omega$,q,s, only.
\end{lemma}
\begin{proof}
Since $q<2$,  by Young's inequality we have
\begin{equation*}
	\lyap{u}\ge\tfrac12 [u]_{s}^2 - 
		\tfrac{\lambda_1(\Omega,q,s)}{4}
		 \left[ \int_\Omega |u|^q\,dx\right]^\frac{2}{q}
		- \tfrac{2-q}{2q} \alpha^\frac{2}{2-q} \left[ \frac{\lambda_1(\Omega,q,s)}{2}
		\right]^{-\frac{q}{2-q}}
\end{equation*}
and then using  \eqref{lambda1} gives \eqref{coercivity}.
\end{proof}

\subsection{The ground state level}
We collect some properties of the minimal energy level, defined as
\begin{equation}
\label{minL}
	\Lambda_1 = \inf_{\varphi\in\sob} \lyap{\varphi}\,.
\end{equation}

\begin{lemma}\label{lm:basic-elliptic}
Let $\alpha>0$ and $1<q<2$. Then,
\begin{itemize}
\item[$(i)$] the energy functional $\lyapno$ achieves the minimum in \eqref{minL};
\item[$(ii)$] the minimum $\Lambda_1$ of $\lyapno$ is a strictly negative number;
\item[$(iii)$] $\lyapno$ has exactly two minimisers $w$ and $-w$, where $w$ is a strictly
positive function;
\end{itemize}
\end{lemma}
\begin{proof}

By \eqref{coercivity}, assertion $(i)$ follows by the compactness of the embedding of $\sob$ into $L^q(\Omega)$.

Then, given any nonzero $\varphi\in\sob$, in view of \eqref{lyap} we have $\lyap{t\varphi}<0$
for $t$ small enough, which implies $(ii)$.

As for $(iii)$, we argue as in \cite[Proposition 2.3]{brafra-20} and we let $u$ be a minimiser. Then $|u|$ is also a minimiser, because
$\lyap{|u|}\le\lyap{u}$ by
the elementary inequality $(|a|-|b|)^2\le(a-b)^2$ with $a=u(x)$,
$b=u(y)$, and thence
\[
\int_{\RN}\int_{\RN} \interr{|u|}{\psi}\, dx\,dy
=\alpha\int_\Omega |u|^{q-2}u\psi\,dx\,,\qquad \text{for all
$\psi \in\sob$.}
\]
Summing the latter to \eqref{LEweak} gives that the positive part $(u+|u|)/2$
of $u$ is a non-negative weak supersolution of \eqref{LEeq}. Then,
it must be either identically zero or strictly positive by the strong maximum principle for
the fractional Laplacian, see {\em e.g.}~\cite[Lemma 6]{serval} or~\cite[Proposition~7.1]{fralic}.
As a consequence, minimisers are
non-negative solutions of \eqref{LEeq}. By~\cite[Proposition 3.4]{fralic} (see also Remark~4.1 therein), non-negative weak solutions of \eqref{LEeq} are unique, and then we deduce $(iii)$.
\end{proof}

\subsection{Higher energies}
We then prove some basic properties of higher energy levels.
\begin{lemma}
Let $1<q<2$ and $\alpha>0$. Then
\label{basic-ell-2}
\begin{itemize}
\item[$(i)$] $\lyapno$ satisfies the Palais-Smale condition;
\item[$(ii)$] $\lyapno$ has the mountain pass structure;
\item[$(iii)$] the critical levels form a compact subset of $[\Lambda_1,0]$.
\end{itemize}
\end{lemma}

\begin{proof}
The first statement is the compactness
in $\sob$ of every Palais-Smale sequence. Then, we
let $(u_n)_{n\in\mathbb N}$ be one. Without loss of generality, that amounts to assuming
that
\begin{subequations}\label{PS}
\begin{equation}
\label{PS1}
\sup_{n\in\mathbb N} \lyap{u_n}<+\infty
\end{equation}
and that, for all $\psi\in\sob$ with $\|\psi\|_{\sob}=1$, we have
\begin{equation}
\label{PS2}
	\left\vert\int_{\RN}\int_{\RN} \interr{u_n}{\psi} 
	- \alpha\int_\Omega |u_n|^{q-2}u_n\psi\,dx \right\vert \le \frac1n\,.
\end{equation}
\end{subequations}
 By
Lemma~\ref{lm:coerc}, Eq.\ \eqref{PS1} implies that
$(u_n)_{n\in\mathbb N}$
is bounded in $\sob$. Thus, thanks to the compactness of the embedding into $L^q(\Omega)$, 
a subsequence of $(u_n)_{n\in\mathbb N}$ (that we do not relabel) converges to some limit $u$ weakly in $\sob$ and
strongly in $L^q(\Omega)$. By passing to the limit in \eqref{PS2} we see\footnote{
Here, in particular, we are using that
\[
	\lim_{n\to\infty}\int_\Omega |u_n|^{q-2}u_n\psi\,dx=\int_\Omega |u|^{q-2}u\psi\,dx\,.
\] 
This can be deduced from the fact that $\||u_n|^{q-2}u_n-|u|^{q-2}u\|_{L^{q/(q-1)}(\Omega)}\to0$, as $n\to\infty$, and in turn this
latter assertion holds because of the convergence to $u$ of the sequence $(u_n)_n$ in $L^q(\Omega)$. To see this, for
$q<2$ one can use the H\"older continuity of the function $\tau\mapsto|\tau|^{q-2}\tau$. 
} that
$u$ is a weak solution of \eqref{LEeq}. Then, we can choose $\psi=u$
in \eqref{LEweak}, which gives
\begin{equation}
\label{nehari}
	\lyap{u} = \left(\frac{1}{2}-\frac{1}{q}\right) [u]_{s}^2
	= \alpha\left(\frac{1}{2}-\frac{1}{q}\right) \int_\Omega|u|^{q}\,dx\,.
\end{equation}
Then,
\[
[u]_{s}^2 = \alpha\int_\Omega |u|^q\,dx
=\lim_{n\to\infty}\alpha \int_\Omega |u_n|^q\,dx\ge \limsup_{n\to\infty}\left([u_n]_{s}^2-\tfrac1n[u_n]_{s}\right)
=\limsup_{n\to\infty}\,[u_n]_{s}^2\,.
\]
Here, we used \eqref{nehari} for the first equality, 
the convergence in $L^q(\Omega)$ for the second one,
Eq.\ \eqref{PS2} for the inequality, and 
Lemma~\ref{lm:coerc} together with \eqref{PS1} for the last equality.
Hence, by the sequential weak lower semicontinuity of $[\,{\cdot}\,]_s^2$
the convergence of the sequence is also strong in $\sob$.
\medskip

We have proved statement $(i)$ and we consider now $(ii)$, which means 
\[
	\inf\left\{ \lyap{\varphi} \mathbin{\colon} \min\big\{\|\varphi-w\|_{\sob}\mathbin{,}\|\varphi+w\|_{\sob}\big\}\ge
	\|w\|_{\sob} \right\}>\Lambda_1\,,
\]
where $w$ is the positive function with energy $\Lambda_1$.
The contrapositive statement is that we have $\lyap{\varphi_j}\to \Lambda_1$ along a sequence
for which is at a distance both to $w$ and to $-w$ larger than the half of that between $w$ and $-w$,
in contradiction with the strong convergence  in $\sob$ either to $w$ or to $-w$, that follows by coercivity (see Lemma~\ref{lm:coerc}).
\medskip

So, $(ii)$ is true and we are left to prove $(iii)$. To do so, we observe that if 
$u\in\sob$ is a critical point of $\lyapno$ then choosing $\psi=u$ in \eqref{LEweak}
yields \eqref{nehari}.
Thus, all energy levels belong to a bounded set, contained in
 $]\Lambda_1,0]$. To prove that that their collection is closed, we write
 \eqref{nehari} with $u$ replaced by $u_n$, an arbitrary sequence of critical points with energy accumulating to a limit value $\Lambda$. Then, by
 the compactness of the embedding of $\sob$
 into $L^q(\Omega)$, 
 the limit $u$ 
 of the sequence in $\sob$ satisfies Eq. \eqref{LEeq}, and
so $\lyap{u}=\Lambda$ by construction.
\end{proof}

\subsection{Mountain pass level}
In the following, we recall how to construct a mountain pass energy level by considering paths of bounded energy
that are continuous with
respect to the topology of the convergence in measure.

\begin{proposition}
\label{prop:mszeta}Set
\[
	\mathfrak{Z} = \left\{z \in C([0,+\infty);L^0(\Omega))\cap L^\infty([0,+\infty);\sob )
	\mathbin{\colon}
	\text{$z(0)=w$ and $\lim_{t\to+\infty}z(t)=-w$ in $L^0(\Omega)$}
	\right\}
\]
Then, for all $1<q<2$ and $\alpha>0$,
\[
	\Lambda^\ast =
	 \inf_{z\in\mathfrak{Z}}\sup_{t\in[0,+\infty)} 
	\lyap{z(\cdot,t)}
\]
is a critical value of $\lyapno$.
\end{proposition}

\begin{proof}
We first notice that $\mathfrak{Z}$ can been replaced with the class
\[
\begin{split}
\mathfrak{Z}_q = 
\Big\{z \in C([0,+\infty);L^q(\Omega))\cap L^\infty([0,+\infty);\sob )
	& \colon 
	\text{
	$z(0)=w$, $\lim_{t\to+\infty}\|z(t)+w\|_{L^q(\Omega)}=0$}
	\Big\}
\end{split}\,.
\]
Indeed $\mathfrak{Z}_q\subset \mathfrak{Z}$ because the convergence in $L^{q}$ implies that in measure. For the reverse inclusion,
we take $z\in \mathfrak{Z}$. Then
$[z(t)]_{s}^2\leq C$, for some constant $C>0$,
for all $t\in [0,+\infty)\setminus \mathcal{N}$, where $\mathcal{N}$ is a negligible set in $[0,+\infty)$. If $\bar{t}\in \mathcal{N}$ and $(t_{n})\subset [0,+\infty)\setminus\mathcal{N}$ converges to $\bar{t}$. Since by Fatou Lemma the nonlocal energy $[\,{\cdot}\,]_{s}^2$ is lower semicontinuous with respect to the convergence in measure and $z(t_n)\to z(\bar t)$ in measure, we have $[z(\bar t)]^2\le C$ as well. Hence, $z$ is equibounded in $\sob$ and, by the compact embedding
in $L^q(\Omega)$, it follows that from every arbitrary sequence $t_n$ converging to any given $t_0\ge0$ we may extract another one along which $z$
converges in $L^q(\Omega)$ to a limit, that must always be $z(t_0)$ because $z(t_n)\to z(t_0)$ in measure. Hence, by the Urysohn property of $L^q(\Omega)$ convergence, we see that 
$z$ is continuous with values in $L^q(\Omega)$.
\medskip

Next, one sees that
\[
 \inf_{z\in\mathfrak{Z}_q}\sup_{t\in[0,+\infty)} 
	\lyap{z(\cdot,t)}
	= \inf_{\gamma\in\mathfrak{G}} \sup_{t\in[0,1]} \lyap{\gamma(t)}\,,
\]
where
\[
	\mathfrak{G}=\Big\{\gamma\in C([0,1];L^q(\Omega))\cap L^\infty([0,1];\sob )\mathbin{\colon}
	\gamma(0)=w\,,\ \gamma(1)=-w\Big\}\,.
\]
This can be seen by repeating verbatim the passages in Part 3 of the proof of~\cite[Theorem 4.2]{bravol}, and we skip the details here.
\medskip

The third step is then to prove that 
\[
	\inf_{\gamma\in\mathfrak{G}} \sup_{t\in[0,1]} \lyap{\gamma(t)}=
	\inf_{\sigma\in \mathfrak{S}} \sup_{t\in[0,1]}\lyap{\gamma(t)}\,,
\]
where 
\[
	\mathfrak{S} = \Big\{\gamma\in C([0,1];\sob)\mathbin{\colon}	\gamma(0)=w\,,\ \gamma(1)=-w\Big\}\,.
\]
We prove this claim by arguing as in Part 2 of the proof of~\cite[Theorem 4.2]{bravol}. More precisely,
 we fix $\varepsilon>0$ and
we take $\gamma_\varepsilon \in C([0,1];L^q(\Omega))$, such that $\gamma_\varepsilon(0)=w$,
$\gamma_\varepsilon(1)=-w$, and 
\begin{equation}
\label{MPqmin}
\sup_{t\in[0,1] } \lyap{\gamma_\varepsilon(t)}<\inf_{\sigma\in \mathfrak{G}} \sup_{t\in[0,1]}\lyap{\gamma(t)} + \varepsilon\,.
\end{equation}
Thus, if we fix $\delta>0$, by uniform continuity there exists $\eta>0$ such that if $|t-s|<\eta$, we have
\[
\|\gamma_\varepsilon(t)-\gamma_\varepsilon(s)\|_{L^q(\Omega)}<\delta.
\]
Now we take a partition $\{t_0,\dots,t_k\}$ of $[0,1]$ such that
\[
t_0=0,\qquad t_k=1,\qquad |t_i-t_{i+1}|<\eta, \mbox{ for every } i=0,\dots,k-1,
\]
then we define the new curve $\theta_\varepsilon:[0,1]\to \sob$, which is given by the piecewise affine interpolation of the points $\gamma_\varepsilon(t_0),\gamma_\varepsilon(t_1),\dots,\gamma_\varepsilon(t_k)$, namely
\[
\theta_\varepsilon(t)=\left(1-\frac{t-t_i}{t_{i+1}-t_i}\right)\,\gamma_\varepsilon(t_i)+\frac{t-t_i}{t_{i+1}-t_i}\,\gamma_\varepsilon(t_{i+1}),\qquad \mbox{ for every } t\in[t_i,t_{i+1}].
\]

Then we take $\delta>0$ and we find a finite increasing sequence of real numbers $t_i\in [0,1]$ such that
for each interval $[t_{i-1},t_i] $ we have $\|\gamma_\varepsilon(t)-\gamma(s)\|_{L^q(\Omega)}\le \delta$ for all $s,t$ in that interval.
Let $\theta_{\varepsilon,\delta}$ denote the  piecewise affine interpolation of the finite sequence of points $\gamma(t_i)$. We set
\[
\tau=\frac{t-t_i}{t_{i+1}-t_i}.
\]
Then, by the standard convexity of the squared seminorm we have
\[
	[\theta_{\varepsilon,\delta}((1-\tau)t_{i-1} + \tau t_i) ]_{s}^2\le 
	(1-\tau)[\gamma_\varepsilon(t_{i-1})]_{s}^2+\tau[\gamma_\varepsilon(t_i)]_{s}^2\,.
\]
Thus, 
\[
	\lyap{\theta_{\varepsilon,\delta}((1-\tau)t_{i-1} + \tau t_i)}
	\le (1-\tau)\lyap{\gamma_\varepsilon(t_{i-1})}+\tau\lyap{\gamma_\varepsilon(t_i)} + \frac{\alpha}{q}(\mathcal{R}_1-\mathcal{R}_2)\,,
\]
where
\[
\mathcal{R}_1=(1-\tau) \int_\Omega|\gamma_\varepsilon(t_{i-1})|^q\,dx +\tau\int_\Omega|\gamma_\varepsilon(t_{i})|^q\,dx \,,\quad
	\mathcal{R}_2 = \int_\Omega |(1-\tau)\gamma_\varepsilon(t_{i-1})+\tau\gamma_\varepsilon(t_{i})|^q\,dx\,.
\]
Now, by using \eqref{MPqmin} we get
\begin{equation}
\label{MPprev}
\lyap{\theta_{\varepsilon,\delta}((1-\tau)t_{i-1} + \tau t_i)}\le 
\inf_{\sigma\in \mathfrak{S}} \sup_{t\in[0,1]}\lyap{\gamma(t)} + \varepsilon + \frac{\alpha}{q}(\mathcal{R}_1-\mathcal{R}_2)\,,
\end{equation}
Also, since the infimum in \eqref{lambda1} is positive, by the coercivity estimate of Lemma~\ref{lm:coerc} and
by \eqref{MPqmin},
\[
	\left(\int_\Omega |\gamma_\varepsilon(t_i)|^q\,dx\right)^\frac{q-1}{q}\le \lambda_1(\Omega,s,q)^\frac{q-1}{2}
	\left[\inf_{\sigma\in \mathfrak{S}} \sup_{t\in[0,1]}\lyap{\gamma(t)} +\varepsilon + C(\Omega,s,q)\right].
\]
Hence, we can infer the estimate $\mathcal{R}_1-\mathcal{R}_2\le C\delta$, with a constant $C$ depending only on the data,
as done in~\cite{bravol}. Inserting this estimate in \eqref{MPprev} we arrive at 
that
\[
\sup_{t\in[0,1]}\lyap{\theta_{\varepsilon,\delta}(t)} \le\inf_{\sigma\in \mathfrak{G}} \sup_{t\in[0,1]}\lyap{\gamma(t)} + \varepsilon+C\delta
\]	
Since the piecewise affine path $\theta_{\varepsilon,\delta}$ is obviously
continuous with values in $\sob$ and $\delta,\varepsilon $ were arbitrary, we deduce that
\[
\inf_{\theta\in\mathfrak{S}}\sup_{t\in[0,1]}\lyap{\theta(t)} \le\inf_{\sigma\in \mathfrak{G}} \sup_{t\in[0,1]}\lyap{\gamma(t)}
\]
The opposite inequality also holds, because  $\mathfrak{S}\subset\mathfrak{G}$, and that ends the third step of the proof.
\medskip

By combining the previous three steps, we see that
\[
\Lambda^\ast = \inf_{\theta\in\mathfrak{S}}\sup_{t\in[0,1]}\lyap{\theta(t)}
\]
and in order to conclude it suffices to prove that the right hand side defines a critical level.
Since by Lemma~\ref{basic-ell-2} the functional $\lyapno$ has a mountain pass structure,
the desired conclusion is therefore a general consequence of~\cite[Chap. \S 2, Theorem 6.1]{S}.
\end{proof}

\subsection{First excited level and spectral gap.}
We set
\begin{equation}
\label{Lambda2}
\Lambda_2=\Lambda_2(s) = \inf\left\{\Lambda>\Lambda_1 \mathbin{\colon}
	\Lambda \text{ is a critical value of } \lyapno\right\}
\end{equation}
We point out that the set in the right hand side is never empty because it always contains $0$, which is the critical value
associated with the critical point $u\equiv0$. Also, its infimum is in fact a minimum because the critical levels form a closed set,
by Lemma~\ref{basic-ell-2}.
\medskip

We first notice that if a gap exists between $\Lambda_1$ and $\Lambda_2$ then it gives room to energy levels.

\begin{proposition} Let $1<q<2$ and $\alpha>0$, and assume that $\Lambda_2>\Lambda_1$. Then there exists a sign-changing function
$\varphi\in\sob$ with $\Lambda_1<\lyap{\varphi}<\Lambda_2$.
\end{proposition}
\begin{proof}
As in~\cite[Proposition 3.5]{bravol}, $\varphi$ will be given by the separate contributions of the least energy solution $w_\varepsilon$ in
\(
\Omega_\varepsilon = \left\{x\in\Omega \mathbin{\colon} {\rm dist}(x,\partial\Omega)>\varepsilon\right\}
\)
and of a function $z_\varepsilon$ supported on a ball $B_{r_\varepsilon}(x_\varepsilon)$ contained in $\Omega\setminus \Omega_\varepsilon$.

We fix $\eta_0\in(0,\Lambda_2-\Lambda_1)$. Arguing as done in~\cite{bravol} we see that for an appropriate $\varepsilon_0>0$ we have
\begin{equation}
\label{fgap1}
\lyap{w_\varepsilon} = \inf\left\{\lyap{u}\mathbin{\colon}u\in W^{s,2}_0(\Omega_\varepsilon)\right\}\le \Lambda_2-\eta_0\,,
\quad \text{for all $\varepsilon\in(0,\varepsilon_0)$.}
\end{equation}
Then, we fix $\psi\in C^\infty_0(B_1)$ and for all $\varepsilon\in(0,\varepsilon_0)$ 
we define
$\psi_\varepsilon(x)=r_\varepsilon^s\psi(\frac{x-x_\varepsilon}{r_\varepsilon})$,
which implies
\begin{equation}
\label{fgap2}
	\lyap{\psi_\varepsilon} = \frac{r_\varepsilon^2}{2}\int_{\RN}\int_{\RN}
	\inter{u}\,dx\,dy - \frac{\alpha r_\varepsilon^{N+q}}{q}\int_{B_1}|\psi|^2\,dx\,.
\end{equation}
Also, we have the general identity
\begin{equation}
\label{fgap3}
\lyap{w_\varepsilon-\psi_{r_\varepsilon}} =\lyap{w_\varepsilon}
+\lyap{\psi_{r_\varepsilon}} +\int_{\Omega_\varepsilon}\int_{B_{r_\varepsilon}(x_\varepsilon)}
	\frac{w_\varepsilon(x)\psi_\varepsilon(y)}{|x-y|^{N+2s}}\,dx\,dy\,.
\end{equation}

Now we take $r_\varepsilon=\varepsilon^2$, so that
\begin{equation}
\label{fgap4}
	\int_{\Omega_\varepsilon}\int_{B_{r_\varepsilon}(x_\varepsilon)}
	\frac{w_\varepsilon(x)\psi_\varepsilon(y)}{|x-y|^{N+2s}}\,dx\,dy =O(\varepsilon^N)\,,\qquad \text{as $\varepsilon\to0^+$.}
\end{equation}
With the choice $r_\varepsilon=\varepsilon^2$, from \eqref{fgap2} we also get
$\lyap{\psi_{r_\varepsilon}}=O(\varepsilon^{2N})$, as $\varepsilon\to0^+$.
Pairing this and \eqref{fgap4}, from the identity \eqref{fgap3} we deduce
that $\lyap{w_\varepsilon-\psi_{\varepsilon^2}}= \lyap{w_\varepsilon}
+O(\varepsilon^N)$, as $\varepsilon\to0^+$. Hence and from \eqref{fgap1}
we infer that the function $\varphi_\varepsilon=w_\varepsilon-\psi_{\varepsilon^2}$
has energy $\lyap{\varphi_\varepsilon}<\Lambda_2$ for
$\varepsilon\in(0,\varepsilon_0) $ small enough. On the other hand, $\varphi \in \sob$ and so $\lyap{\varphi_\varepsilon}\ge\Lambda_1$, and
the inequality must be strict because of assertion $(iii)$ in Lemma~\ref{lm:basic-elliptic}.
\end{proof}

Now we recall a general consequence of the uniqueness of (positive) energy minimizing functions: if there is a spectral gap, then
the mountain pass does not collapse on the global minimum.

\begin{proposition}
\label{prop:MPexcited}
Under the assumptions of Proposition~\ref{prop:mszeta}, we have $\Lambda^\ast\ge\Lambda_2$.
\end{proposition}

\begin{proof} Either $\Lambda_2=\Lambda_1$, and then there is nothing to prove, or else $\Lambda_2>\Lambda_1$.
In arguing by contradiction, we then assume $\Lambda_2>\Lambda_1$ and $\Lambda^\ast=\Lambda_1$. Then
we can find a sequence of admissible paths $z_j$ for the definition of $\Lambda^\ast$ such that
$\sup_{t\ge0} \lyap{z_j(\cdot,t)}<\Lambda_1+2^{-j}$. Since each $t\mapsto z_j(\cdot,t)$ is continuous 
with values in $L^0(\Omega)$, there exists two sets $B,B'\subset \sob$
that are disjoint and open with respect to the (metrizable) topology of the convergence in measure, with $w\in B$ and $-w\in B'$, such that
for every $j $ we find $t_j>0$ with $\lyap{z_j(\cdot,t_j)}<\Lambda_1+2^{-j}$. Then, the functions $z_j(\cdot, t_j)$ would form a minimizing sequence for $\lyapno$, in contradiction with the fact that
they all are bounded away both from the minimizers $w,-w$ in $L^0(\Omega)$.
\end{proof}

We end this section by recalling an important consequence of the stability of energy minimizing solution of the elliptic problem,
which in turn implies the fundamental gap, proved in~\cite{fralic} for the non-local problem following the method
used in~\cite{bradepfra} for the local one. The following statement differs little from the original one in~\cite{fralic},
which simply states the isolation with respect to the $L^1(\Omega)$ topology instead of the topology of the convergence in measure.

\begin{theorem}\label{isolated}
Let $1<q<2$ and $\alpha>0$, and let $w$ be the positive minimizer of $\lyapno$. 
Assume that $\Omega$ is a bounded $C^{1,1}$ open set. Then, the set $\{w,-w\}$ is bounded away in $L^0(\Omega)$ from any critical point
of $\lyapno$.
\end{theorem}
\begin{proof}
We argue by contradiction and we assume that a sequence of critical points $u_j$ of $\lyapno$ converges
in measure to $w$. Since critical energies are negative (see statement $(iii)$ in Lemma~\ref{basic-ell-2})
and the energy is coercive (see Lemma~\ref{lm:coerc}), the sequence belongs to a bounded subset of $\sob$.
Then, by the compactness of the embedding into $L^q(\Omega)$, the sequence converges to $w$ also in $L^q(\Omega)$,
which contradicts the conclusion of~\cite[Proposition 6.1]{fralic}.
\end{proof}
Arguing similarly as in \cite[Lemma 3.4]{bravol} one can observe that if $\left\{\varphi_{n}\right\}$ is minimizing sequence of  $\lyapno$, then it converges, up to subsequence, either to the positive minimizer $w$ or to $-w$. 
We have then the following direct consequence of Theorem \ref{isolated}, namely the fundamental gap between the first and the second critical energy level of the functional $\lyapno$. The proof is similar to the relevant one in \cite[Proposition 3.5]{bravol}.
\begin{corollary}\label{cor:FG}
Let $1<q<2$ and $\alpha>0$. If $\Omega$ is a bounded $C^{1,1}$ open set, then $\Lambda_2>\Lambda_1$.
\end{corollary}

\section{The (rescaled) parabolic problem}\label{sec:4}
Given a solution $u$ of \eqref{FPMEintro}, the rescaled function $v(x,t)=e^{\alpha t} u(x,e^t-1)$ solves the following Dirichlet initial boundary value problem
\begin{equation}\label{FPMEs}
\left\{\begin{array}{rcll}
\partial_{t}v&=&-(-\Delta)^{s} \Phi(v)+\alpha v, & \mbox{ in } Q:=\Omega\times (0,+\infty),\\
v&=&0, & \mbox{ in } \R^{N}\setminus\Omega\times(0,+\infty),\\
v(\cdot,0)&=&u_0,& \mbox{ in } \Omega.
\end{array}
\right.
\end{equation}
We recall here the definition of weak solutions of the rescaled fractional porous media equation problem to \eqref{FPMEintro}.
\begin{definition}\label{energyweak}
Given $T\in(0,+\infty]$ and given
$u_0\in L^{m+1}(\Omega)$, with $\Phi(u_0)\in \sob$, a function
\begin{equation}
\label{asspt:enweak}
\text{$v\in C([0,T];L^{m+1}(\Omega))$, with $\Phi(v)\in L^2((0,T);\sob)$,}
\end{equation}
is said to be a weak solution of \eqref{FPMEs} in
$Q_T=\Omega\times(0,T)$ if the integral equation
\begin{equation}
\label{FMPEsW}
\begin{split}
-\iint_{Q_T} v\ \frac{\partial \eta}{\partial t}\,dx\,dt & +\int_0^T \int_{\mathbb R^{N}}\int_{\mathbb R^N} \frac{(\Phi(v(x,t))-\Phi(v(y,t)))(\eta(x,t)-\eta(y,t))}{|x-y|^{N+2s}}\,dx\,dy\,dt
\\ & =\alpha\iint_{Q_T} v\ \eta\,dx\,dt
\end{split}
\end{equation}
holds for all $\eta\in C_{c}^\infty(Q_{T})$.
\end{definition}

\begin{remark}
Under the {\em a priori} assumption \eqref{asspt:enweak} made in Definition~\ref{energyweak}, weak solutions are often called {\em energy weak solution} in the literature~\cite{DQRV}. Also, as done in~\cite{DQRV}, it is possible to prove that weak solutions are {\em strong}, with a number of implications ({\em e.g.}, $L^1$ contractivity, comparison principles, and more), but in this paper we can limit our attention to (energy) weak solutions. 
\end{remark}

\subsection{Well-posedness and entropy-entropy dissipation inequality}
The inequality proved in the following theorem is a crucial ingredient for the proof of our main result.

\begin{theorem}\label{exist-ima}
Let $u_0\in L^{m+1}(\Omega)$, with $\Phi(u_0)\in \sob$. Then, there exists a unique weak solution
of \eqref{FPMEs} in $Q$ with $v(0)=u_0$. Moreover,
the estimate
\begin{equation}
\label{EED}
\lyapm{\frac{m+1}{m}}{\Phi(v(\cdot,t))}+C_{0} \int_0^t\int_\Omega \left\vert\partial_t g(v)\right\vert^2\,dx\,dy\le
\lyapm{\frac{m+1}{m}}{\Phi(u_0)}
\end{equation}
holds for all $t>0$, for some positive constant $C_{0}=C_{0}(m)$, depending only on $m$, where $g$ is the nonlinearity given by
\[
g(v)=|v|^{\frac{m-1}{2}}v.
\]

\end{theorem}
\begin{proof}
We fix $T>0$. 
We shall prove the existence of a weak solution
$v\in C([0,T];L^{m+1}(\Omega))$ of \eqref{FPMEs}, with $v(0)=u_0$, 
such that the estimate \eqref{EED} is valid for all $0\le t\le T$. Its uniqueness follows by that of {\em weak energy solutions} of the homogeneous equation \eqref{FPMEintro}. Indeed, if $v$ solves \eqref{FPMEs}, then 
$$u(x,t)=(1+t)^{-\alpha}v(x,\log(1+t))$$ solves \eqref{FPMEintro}, and this solution can be seen to
be unique by adapting a method due to
Ole\u{\i}nik, Kala\v{s}nikov, and \v{C}\v{z}ou that consists in inserting a specific test function 
in the weak formulation, so as to obtain an a priori identity implying that the difference between two weak solutions must be zero.
That is done in~\cite[Theorem~6.1]{DQRV} in the case $\Omega=\mathbb R^N$; adapting
the proof to the case of a bounded domain is straightforward.

To prove the existence of a solution with the energy estimate, we fix $h>0$ and 
for all integers $k$ from $1$ to the integer part $\left\lfloor T/h\right\rfloor$ of $T/h$
we define recursively $v_k$ as a solution of the minimization
\begin{equation}
\label{defukh}
	\min \left\{
		\lyapm{\frac{m+1}{m}}{\Phi(v)} + \frac{1}{h}\varphi(v\mathbin{,} v_{k-1})\mathbin{\colon}
		v\in \sob
	\right\}\,,
\end{equation}
where
\begin{equation}
\label{defvarphi}
	\varphi(u,f) = \frac{1}{m+1}\int_\Omega (|f|^{m+1}-|v|^{m+1})\,dx + \int_\Omega \Phi(v)(v-f)\,dx\,.
\end{equation}
Note that $\varphi(u,f)\ge0$, with equality only if $u=f$. 
The Euler-Lagrange equation for \eqref{defukh} is
\begin{equation}\label{Eulerlagr}
(-\Delta)^{s}\Phi(v_{k})+\frac{1}{h}(v_{k}-v_{k-1})=\alpha v_{k}\,,
\end{equation}
{\em i.e.}, the discretized version of \eqref{FPMEs}.
Then, we construct $\bar v_h\colon Q_T\to\mathbb R$ by setting 
\begin{equation}
\label{defubarh}
	\bar v_h (x,t) = v_{\left\lfloor t/h\right\rfloor}(x)\,,
\end{equation}
for all $(x,t)\in Q_T$. The minimality of $v_k$ for the problem \eqref{defukh} implies
\begin{equation}
\label{ee1}
		\lyapm{\frac{m+1}{m}}{\Phi(v_k)} + \frac{1}{h}\varphi(v_k\mathbin{,}v_{k-1})\le	\lyapm{\frac{m+1}{m}}{\Phi(v_{k-1})}\,.
\end{equation}
By Lemma~\ref{lm:A2} with $a=u_{k-1} $ and $b=u_k$, there is a constant $C_0(m)$ depending only on $m$ with
\begin{equation}
\label{ee1.5}
\frac{1}{h} \varphi(v_k\mathbin{,}v_{k-1}) \ge C_0(m) h \int_\Omega \left\vert \frac{g(v_k)-g(v_{k-1})}{h}\right\vert^2\,dx\,.
\end{equation}
Besides \eqref{defubarh}, we define $\bar v_h$ for negative times by setting $\bar v_h(\cdot,t)= u_0$, for all $t<0$,
and we denote by $\widehat G_h$ the backward Steklov average %
of $g\circ \bar v_h$, namely
\[
	\widehat G_h(x,t)=\frac{1}{h}\int_{t-h}^tg(\bar v_h(x,\tau))\,d\tau\,,\qquad \text{for all $(x,t)\in Q_T$.}
\]
Then, by inserting \eqref{ee1.5} in \eqref{ee1} and summing up, we arrive at the energy estimate
\begin{equation}
\label{ee2}
	\lyapm{\frac{m+1}{m}}{\Phi(\bar v_h(\cdot,t))} + 
	 C_0(m)\iint_{Q_t}\! \!\left\vert\partial_t \widehat G_h(x,\tau)\right\vert^2\!\!dx\,d\tau
	\le	\lyapm{\frac{m+1}{m}}{\Phi(u_{0})},\ \text{for all $0\le t\le T$.}
\end{equation}

 Incidentally, by Fubini theorem and Jensen's inequality we have
\[
\begin{split}
	\iint_{Q_T} \big(\widehat G_h -g(\bar v_h)\big)^2\,dx\,dt
	&\le \frac1h \int_0^T\int_{t-h}^t \int_\Omega \big( g(\bar v_h(x,\tau)-g(\bar v_h(x,t))\big)^2\,dx\,d\tau\,dt\\
	&=\sum_{k=1}^{\left\lfloor T/h\right\rfloor}\int_\Omega \big( g(u_k)-g(u_{k-1})\big)^2\,.
\end{split}
\]
Thus, owing to the definition of $\bar v_h$, and by using \eqref{ee1.5} backward, we see that
\begin{equation}
\label{ee2.5}
	\iint_{Q_T} \big(\widehat G_h -g(\bar v_h)\big)^2\,dx\,dt\le C_0^{-1} h \ \Big[ 	\lyapm{\frac{m+1}{m}}{\Phi(u_{0})}+C\Big]
\end{equation}
where $C>0$ is the constant of Lemma~\ref{lm:coerc}.
\medskip

We will make use of Eq.\ \eqref{ee2.5} below. But before doing so, we first aim at proving that
\begin{equation}
\label{preaa}
\text{$\{ \widehat G_h \mathbin{:} h>0 \}$ is relatively compact in $C([0,T];L^2(\Omega))$.}
\end{equation}
To do so, we argue similarly as done in the proof of  \cite[Proposition~5.3,~Step~4]{bravol}. We first 
note that
\[
	\int_\Omega |\widehat G_h(x,t_1)-\widehat G_h(x,t_2)|^2\,dx
	\le (t_2-t_1)\int_{t_1}^{t_2} \int_\Omega |\partial_t \widehat G_h(x,\tau)|^2\,dx\,d\tau
\]	
and, thanks to \eqref{ee2} and Lemma~\ref{lm:coerc}, we also obtain
\[
(t_2-t_1)\int_{t_1}^{t_2} \int_\Omega |\partial_t \widehat G_h(x,\tau)|^2\,dx\,d\tau
	\le C_0^{-1}|t_2-t_1| \Big[\lyapm{\frac{m+1}{m}}{\Phi(u_{0})} +C\Big]\,.
\]
Putting the last two estimates together entails that
\begin{equation}
\label{preaa1}
\text{$\{ \widehat G_h :h>0\} $ is an equicontinuous family of
$L^{2}(\Omega)$-valued curves.}
\end{equation}

Then, we fix $0\le t\le T $ and $z\in\mathbb R^N$, and we use Jensen's inequality to get
\[
	\int_{\R^{N}} \left|\widehat G_h(x+z,t)-\widehat G_h(x,t)\right|^\frac{4m}{m+1}\,dx\le
	\frac{1}{h}\int_{t-h}^{t}\int_{\R^{N}} \left| g(\bar{v}_{h}(x+z,\tau))-g(\bar{v}_{h}(x,\tau))\right|^\frac{4m}{m+1}dx\,d\tau\,.
\]
In view of the elementary inequality $|g(A)-g(B)|^\frac{4m}{m+1}\le C_1(m)|\Phi(A)-\Phi(B)|^2$ and of~\cite[Lemma~A.1]{BLP}
\[
\int_{\R^{N}} \left| g(\bar{v}_{h}(x+z,\tau))-g(\bar{v}_{h}(x,\tau))\right|^\frac{4m}{m+1}dx\le C_2(m,N)
|z|^{2s} 
[\Phi(\bar v_h(\cdot,\tau))]^2\,,\qquad \text{for all $t-h<\tau<t$.}
\]
Inserting this inequality into the previous one and using H\"older inequality to handle the result gives
\[
	\int_{\R^{N}} \left|\widehat G_h(x+z,t)-\widehat G_h(x,t)\right|^2\,dx\le
	|z|^{\frac{(m+1)s}{m}}\cdot\left[C_3(m,N)\ |\Omega|^{\frac{m-1}{2m}}\ \frac1h \int_{t-h}^t[\Phi(\bar v_h(\cdot,\tau))]_{W_{0}^{s,2}(\Omega)}^\frac{m+1}{m}\,d\tau\right]\,.
\]
Also, by Lemma~\ref{lm:coerc} and by Eq.\ \eqref{ee1} we have
\[
 \int_{t-h}^t[\Phi(\bar v_h(\cdot,\tau))]_{W_{0}^{s,2}(\Omega)}^\frac{m+1}{m}\,d\tau \le h\cdot C_4(s,m,\Omega)\ \left[ 
 	\lyapm{\frac{m+1}{m}}{\Phi(u_0)}+1\right]^{\frac{m+1}{2m}}\,.
\]
\begin{subequations}\label{preaa2}
Recalling that $t\ge0$ and $z\in\mathbb R^N$ were arbitrary, the last two inequalities entail that
\begin{equation}
\label{preaa2.a}
\lim_{z\to0} \sup_{h>0} 	\int_{\R^{N}} \left|\widehat G_h(x+z,t)-\widehat G_h(x,t)\right|^2\,dx=0\,,\qquad \text{for all $0\le t\le T$.}
\end{equation}
Since $g(\sigma)^2=|\Phi(\sigma)|^\frac{m+1}{m}$, by the fractional Sobolev embedding into $L^{\frac{m+1}{m}}(\Omega)$ 
we have
\begin{equation*}
\int_\Omega |\widehat G_h(x,t)|^2\,dx \le \frac1h\int_{t-h}^t \int_\Omega |g(\bar v_h(x,\tau))|^2\,dx\,d\tau\le
  \frac{C}{h}\int_{t-h}^t [\Phi(\bar v_h(\cdot,\tau))]^2\,d\tau
\end{equation*}
and thence, by arguing as done above, we infer that
\begin{equation}
\label{preaa2.b}
	\sup_{h>0} \int_\Omega |\widehat G_h(x,t)|^2\,dx<+\infty\,,\qquad\text{for all $0\le t\le T$.}
\end{equation}
\end{subequations}
By Fr\'echet-Kolmogorov theorem, \eqref{preaa2} implies that
\begin{equation}
\label{preaa3}
\text{$	\{ \widehat G_h(\cdot,t)\mathbin{\colon} h>0\}$ is relatively compact in $L^2(\Omega)$, for all $0\le t\le T$.}
\end{equation}
Thanks to the vector-valued extension of Ascoli-Arzelà theorem~\cite[Lemma~1]{Si}, from
\eqref{preaa1} and \eqref{preaa3} we can infer \eqref{preaa}. 

\begin{subequations}\label{ee3}
Then, there exist
a sequence $h_j\to0^+$ and a function $v$, with $g(v)\in C([0,T];L^2(\Omega))$, such that

\begin{equation}
\label{postaa1}
	\text{$\widehat G_{h_j} \to g(v) $ \qquad in $C([0,T];L^2(\Omega))$.}
\end{equation}
Clearly, the convergence \eqref{postaa1} is strong in $L^2(Q_T)$, too. Hence, recalling \eqref{ee2.5}, we deduce that
\begin{equation}
\label{postaa2}
	\text{$ g(\bar v_{h_j}) \to g(v) $ \qquad strongly in $L^2([0,T];L^2(\Omega))$.}
\end{equation}
Then, by Ineq.\ (A.2) in~\cite[Lemma A.1]{bravol}, it follows that
\begin{equation}
\label{postaa3}
\text{$ \bar v_{h_j} \to v $ \qquad strongly in $L^{m+1}([0,T];L^{m+1}(\Omega))$.}
\end{equation}
Also, by possibly passing to a subsequence, from Lemma~\ref{lm:coerc} and Eq.\ \eqref{ee1} we may infer that
\begin{equation}
\label{postaa4}
\text{$ \Phi(\bar v_{h_j}) \to \Phi(v) $ \qquad weakly in $L^2(0,T;\sob)$.}
\end{equation}
Moreover, we note that $(\partial_t \widehat G_{h_j})_j$ is a bounded sequence in $L^2(Q_T)$ because of estimate \eqref{ee2}. Thus,
in view of \eqref{postaa1}, up to passing to a further subsequence, we may write that
\begin{equation}
\label{postaa5}
\text{$ \partial_t \widehat G_{h_j} \to \partial_t g(v) $ \qquad weakly in $L^2(0,T;L^2(\Omega))$.}
\end{equation}
\end{subequations}

By lower semicontinuity,
\eqref{ee2} and \eqref{ee3} imply \eqref{EED}.
Since $\widehat G_{h_j}(0)=g(u_0)$ for all $j$, by \eqref{postaa1} we also have that $v(0)=u_0$.
Then, in order to conclude we are left to prove that $v$ is a weak solution of \eqref{FPMEs} in $Q_T$.
To see this, we first deduce \eqref{asspt:enweak} from \eqref{postaa1}, thanks to~\cite[Lemma A.1]{bravol}.
To prove that \eqref{FMPEsW} holds too, we use the Euler-Lagrange equation \eqref{Eulerlagr} for $v_k$ and \eqref{defubarh} to get
\[
\begin{split}
		\iint_{Q_T}   & \frac{\bar v_{h_j}(x,t)-\bar v_{h_j}(x,t-h_j)}{h_j}\ \eta(x,t)\,dx\,dt   \\
		& +
	\int_0^T\int_{\mathbb R^N}\int_{\mathbb R^N}
	\frac{\Phi(\bar v_{h_j}(x,t))-\Phi(\bar v_{h_j}(y,t))}{|x-y|^{N+2s}}(\eta(x,t)-\eta(y,t))\,dx\,dy\,dt=\alpha\iint_{Q_T} \bar v_{h_j} \eta\,,
\end{split}
\]
for all $\eta\in C^\infty_c(Q_{T})$, and changing variables 
yields
\[
-\iint_{Q_T}  \bar v_{h_j}\ \partial_t\widehat\eta^{h_j} +\int_0^T\!\!\iint_{\mathbb R^{2N}}\!\!
	\frac{\Phi(\bar v_{h_j}(x,t))-\Phi(\bar v_{h_j}(y,t))}{|x-y|^{N+2s}}(\eta(x,t)-\eta(y,t))\,dx\,dy\,dt=\alpha\iint_{Q_T} \bar v_{h_j} \eta\,.
\]
Thanks to \eqref{ee3}, by passing to the limit in the latter we obtain Eq.\ \eqref{FMPEsW} and we conclude.
\end{proof}

\subsection{Stabilization} We characterise the cluster points of large time asymptotic profiles of weak solutions, understood 
as in the following definition. 

\begin{definition}
Let $u_0\in L^{m+1}(\Omega)$, with $\Phi(u_0)\in \sob$. Then, the {\em $\omega$-limit emanating from $u_0$} is the set
\[
	\omega(u_0)=\Big\{
	f\in L^{m+1}(\Omega) \mathbin{\colon}
	 \text{there exists $(t_j)_j\nearrow+\infty$ with $\displaystyle\lim_{j\to\infty}\|v(\cdot,t_j)-u_0\|_{L^{m+1}(\Omega)}=0$}
	\Big\}
\]
where $v\in C([0,\infty);L^{m+1}(\Omega))$ is the weak solution of \eqref{FPMEs} with $v(0)=u_0$.
\end{definition}

The structure of $\omega(u_0)$ is easier to understand under the assumptions 
\begin{equation}
\label{hp-omega}
\partial_tg(v)\in L^2([T_0,+\infty),L^2(\Omega))\,,\qquad \Phi(v)\in L^\infty( [T_0,+\infty),\sob)
\end{equation}
on the weak solution $v$ of \eqref{FPMEs} with initial datum $u_0$, for an appropriate time $T_0>0$. These assumptions
are the non-local counterpart of  those considered in~\cite{bravol} for the local case.

\begin{theorem}\label{omega-lim-char}
Let  $v$ 
be the weak solution of \eqref{FPMEs} and assume that there exists $T_0>0$ for which \eqref{hp-omega}
holds.
 Then, for every $U\in\omega(u_0)$,
the function $\Phi(U)$ belongs to $\sob$ and is a weak solution of \eqref{LEeq}.
\end{theorem}
\begin{proof}
By repeating verbatim the proof of~\cite[Theorem 5.2]{bravol}, we can see that the assumptions \eqref{hp-omega} imply
the first statement and, also, we arrive at
\begin{equation}
\label{bravol5.2}
\lim_{j\to\infty}\| v_j-U\|_{L^{m+1}(\mathcal Q)} =0\,,\quad \text{where $\mathcal{Q}=\Omega\times(-1,1)$,}
\end{equation}
and, for all $j\in\mathbb N$, we set $v_j(x,t)=v(x,t+t_j)$,
for all $(x,t)\in \mathcal{Q}$. 

In order to prove also that $\Phi(U)$ is a weak solution of \eqref{LEeq},
we follow~\cite{bravol}, again: we take $\rho\in C_0^\infty(-1,1)$ and $\psi\in C^\infty_0(\Omega)$, and we test Eq.\ \eqref{FMPEsW} with $\eta(x,t)=\rho(t-t_j)\psi(x)$,
so as to get
\[
\begin{split}
-\int_{-1+t_j}^{1+t_j}\int_{\Omega} v \psi \rho'(t-t_j)\,dx\,dt & +\int_{-1+t_j}^{1+t_j} \iint_{\mathbb R^{2N}}\!\!\!\! \frac{(\Phi(v(x,t))-\Phi(v(y,t)))(\psi(x)-\psi(y))}{|x-y|^{N+2s}}\,dx\,dy
\rho(t-t_j)\,dt
\\ & =\alpha\int_{-1+t_j}^{1+t_j}\int_{\Omega} v\psi\rho(t-t_j)\,dx\,dt\,.
\end{split}
\]
A change of variable in the time integral yields
\[
\begin{split}
-\iint_{\mathcal{Q}}\!\! v_j \psi \rho'\,dx\,dt & +\int_{-1}^1 \iint_{\mathbb R^{2N}}\!\! \frac{(\Phi(v_j(x,t))-\Phi(v_j(y,t)))(\psi(x)-\psi(y))}{|x-y|^{N+2s}}\,dx\,dy
\,
\rho(t)\,dt \\ & =\alpha\iint_{\mathcal{Q}} v_j\psi\rho\,dx\,dt\,.
\end{split}
\]
In view of the definition of the $s$-laplacian of the smooth function $\psi$, the latter implies
\[
\begin{split}
-\iint_{\mathcal{Q}}\!\! v_j \psi \rho'\,dx\,dt & +
\iint_{\mathcal{Q}} \Phi(v_j) (-\Delta)^s\psi\ \rho\,dt  =\alpha\iint_{\mathcal{Q}} v_j\psi\rho\,dx\,dt\,.
\end{split}
\]
Owing to~\eqref{bravol5.2}, taking the limits yields
\[
	\begin{split}
-\int_{-1}^1 \left(\int_{\Omega} U\psi\,dx\right) \rho'\,dt & +
\int_{-1}^1\left(\int_{\Omega} \Phi(U) (-\Delta)^s\psi\,dx\right) \rho\,dt  =\alpha\int_{-1}^1\left(\int_{\Omega} U\psi\,dx\right)\rho\,dt\,.
\end{split}
\]
Since $\rho$ vanishes at the endpoints of the interval $(-1,1)$, it follows that
\[
\left[\int_{\Omega}\left( \Phi(U) (-\Delta)^s\psi- \alpha U\psi\right)\,dx\right]\cdot\int_{-1}^1\rho\,dt=0\,.
\]
As $\rho$ can be any element of $C^\infty_0(-1,1)$, we can choose is so as to make the time integral different from zero. Thus, recalling again the definition of $(-\Delta)^s\psi$ we deduce
that \eqref{LEweak} holds with $u=\Phi(U)$, as desired.
\end{proof}

\section{Paths of controlled energy}

\begin{proposition}\label{prop:encontr}
Let $m>1$ and $\alpha>0$, let $u_0\in L^{m+1}(\Omega)$, with $\Phi(u_0)\in \sob$.
Assume that either \eqref{assnl1} or \eqref{assnl2} holds, and set $q=(m+1)/m$.
Then, there exists $\theta\in C([0,1]\mathbin{;} \sob)$ for which
\begin{itemize}
\item[$(i)$]$\theta(\cdot,0)$ is the positive minimizer $w$ of $\lyapno$, 
\item[$(ii)$] $\theta(\cdot,1) = \Phi(u_0)$, and
\item[$(iii)$]   $\lyap{\theta(\cdot,t)}<\Lambda_2$ for all $t\in(0,1)$.
\end{itemize}
\end{proposition}
\begin{proof}
In order to construct the desired function $\theta$, we first consider a special path $\gamma$ in $\sob$, 
connecting $w$ and the positive part $\Phi(u_0)^+$ of $\Phi(u_0)$. This is done by setting
\[
	\gamma(\tau) = \Big[ (1-\tau) w^q + \tau |\Phi(u_0)^+|^q\Big]^\frac{1}{q}\,, \qquad \text{ for every $\tau\in[0,1]$.}
\]
By~\cite[Proposition 4.1]{brafra-14}, $\tau\mapsto [\gamma(\tau)]_{s,\Omega}^2$ is convex.
In particular, it is continuous and it
follows that $\gamma$ is continuous with values in $\sob$. Also, recalling \eqref{lyap}, we have
\begin{equation}
\label{control1}
\lyap{\gamma(\tau)}\le (1-\tau)\lyap{w}+\tau\lyap{\Phi(u_0)^+}\,, \qquad \text{ for all $\tau\in[0,1]$.}
\end{equation}
Under either of the assumptions \eqref{assnl1} and \eqref{assnl2} that implies
\begin{equation}
\label{control1bis}
\lyap{\gamma(\tau)}< \Lambda_2\,,\qquad \text{for all $\tau\in[0,1]$.}
\end{equation}

Then, we consider the segment in $\sob$ with endpoints $\Phi(u_0)^+$ and $\Phi(u_0)$. The
linear parametrization of such segment, defined by $\sigma(\tau)=\Phi( u_0)^+-\tau \Phi(u_0)^-$ is obviously continuous with values in $\sob$. Also, we have
\[
	\lyap{\sigma(\tau)} = h(\tau) + C\tau+K
\]
where
\[
	C = 2\iint\frac{\Phi(u_0)^+(x)\Phi(u_0)^-(y)}{|x-y|^{N+2s}}\,dx\,dy\,,
	\qquad K 
	=\lyap{\Phi(u_0^+)}\,,
\]
and $h(\tau)$ is essentially the function considered in Appendix to \cite{bravol} in the local case. Namely, 
\[
h(\tau)=A\tau^2-B\tau^q\,,\qquad\text{where } A = \frac12[\Phi(u_0)^-]_{s,\Omega}^2\quad\text{and } B=\frac{\alpha}{q}\int_\Omega |\Phi(u_0)^-|^q\,dx\,.
\]
If $\tau_0:= \left(\frac{qB}{2A}\right)^\frac{1}{2-q}\ge1$, then by direct inspection $h'(\tau)<0$ for $\tau\in(0,1)$, and hence
\begin{equation}
\label{topF}
	\lyap{\sigma(\tau)}\le C+K =\lyap{\Phi(u_0^+)}+
2\iint\frac{\Phi(u_0)^+(x)\Phi(u_0)^-(y)}{|x-y|^{N+2s}}\,dx\,dy\,.
\end{equation}
Otherwise, $qB<2A$ and that implies $h'(1)>0$. Also, for $\tau_0\le \tau\le 1$ we have
$h''\le -2Aq<0$, so that $h'(\tau)\ge h'(1)>0 $ for $\tau\in[\tau_0,1]$. If instead
$0\le \tau<\tau_0$ then $h'(\tau)<0$, because of the definition of $\tau_0$. Therefore,
for all $\tau\in[0,1]$ the inequality
\[
\begin{split}
	\lyap{\sigma(\tau)} &\le \max \{h(0)+C+K\mathbin{,} h(1)+C+K\}
	\\ & =
	\max\left\{	\lyap{\Phi(u_0^+)}+2\iint\frac{\Phi(u_0)^+(x)\Phi(u_0)^-(y)}{|x-y|^{N+2s}}\,dx\,dy	\mathbin{,}	\lyap{\Phi(u_0)}	\right\}
\end{split}
\]
holds regardless of the value of $\tau_0$. Under either of the assumptions in \eqref{assumnonloc}, that entails
\begin{equation}
\label{control2}
\lyap{\sigma(\tau)} < \Lambda_2 \,,\qquad\text{for all $\tau\in[0,1]$.}
\end{equation}

By construction, setting
\[
	\theta(t)= \begin{cases}
		\gamma(2t)\,, & \qquad\text{if $0\le t<\frac12$,}\\
		\sigma(2(t-\tfrac12))\,,&\qquad\text{if $\frac12\le t<1$} 
	\end{cases}
\]
defines a continuous function from $[0,1]$ to $\sob$ for which the assertions $(i)$ and $(ii)$ are true. As for $(iii)$, that
is a consequence of the inequalities \eqref{control1bis} and \eqref{control2}.
\end{proof}

\begin{remark}\label{lims1}If $\varphi\in W^{1,2}_0(\Omega)$ then
\[
\lim_{s\nearrow1}(1-s)\iint\frac{\varphi^+(x)\varphi^-(y)}{|x-y|^{N+2s}}\,dx\,dy=0\,.
\]
That is a consequence of the known fact, see {\em e.g.}~\cite[Corollary 3.20]{EE}, that 
\[
	\lim_{s\nearrow1} (1-s)[\varphi]_s^2 = C(n) \int_\Omega|\nabla\varphi|^2\,dx\,,
\]	
and of the locality of Sobolev seminorm in the right hand side of the latter.
With $\varphi = \Phi(u_0)$, we see that the double integral in \eqref{assnl2} vanishes in the limit up 
to multiplying it by the factor $(1-s)$.
\end{remark}

\section{Proofs of the main results}
\subsection{Proof of Theorem~\ref{mainthm1}}Weak solutions can be defined for \eqref{FPMEintro} similarly as done in Definition~\ref{energyweak} for the
rescaled problem \eqref{FPMEs}. By setting
\[
v(x,t)=e^{\alpha t}u(x,e^t-1)
\]
the desired conclusion becomes that the weak solution $v(\cdot,t)$ of \eqref{FPMEs} with $v(0)=u_0$ converges, as $t\to+\infty$, either to $\Phi^{-1}(w)$ or to $-\Phi^{-1}(w)$ in $L^{m+1}(\Omega)$,
where $w$ is the positive minimiser of the functional 
$\lyapno$  defined by \eqref{lyap}, that is unique by Lemma~\ref{lm:basic-elliptic}.
By Theorem~\ref{exist-ima}, $v$ is uniquely determined and the estimate \eqref{EED} holds. Therefore, by the compactness of the embedding
$\sob$ into $L^{q}(\Omega)$, it follows that the orbit $\{v(\cdot,t)\colon t>0\}$
is precompact in $L^{m+1}(\Omega)$. Then, the omega-limit $\omega(u_0)$ is connected,
and so in order to get the desired conclusion
it suffices to prove that
\begin{equation}
\label{fin}
\omega(u_0)\subseteq\{\Phi^{-1}(w),-\Phi^{-1}(w)\}\,.
\end{equation}

Then, we take $U\in\omega(u_0)$.
From \eqref{EED} 
we can infer \eqref{hp-omega} and so, by Theorem~\ref{omega-lim-char},  $\Phi(U)$ is 
a critical point of 
$\lyapno$.
Therefore, it is enough to make sure that
\begin{equation}
\label{fineq}
\lyap{\Phi(U)}<\Lambda_{2}
\end{equation}
because by Corollary~\ref{cor:FG} that entails
that $\Phi(U)$ is either $w$ or $-w$, which in turn gives \eqref{fin}.

We are left to prove \eqref{fineq}. To do so, we take a sequence $t_j\nearrow+\infty$ with $v(\cdot,t_j)\to U$ in $L^{m+1}(\Omega)$. We raise
the Lipschitz estimate $|\Phi(b)-\Phi(a)|\le m (|a|\vee|b|)^{m-1} |b-a|$, with
$a=U(x)$ and $b=v(x,t_j)$, to the power $q=\frac{m+1}{m}$, we integrate the result over $\Omega$, and hence we arrive at
\[
\begin{split}
	\limsup_{j\to\infty} \|\Phi(v(\cdot,t_j)) & -\Phi(U)\|_{L^q(\Omega)}
	\le m\Big[\|\Phi(U)\|_{L^q(\Omega)}+\sup_{t>0} \|v(\cdot,t)\|_{L^{m+1}(\Omega)}\Big] \lim_{j\to\infty}\| v(\cdot,t_j)-U\|_{L^{m+1}(\Omega)}\,,
\end{split}
\]	
where we also used H\"older inequality with exponents $m/(m-1)$ and $m$.
Hence, $\Phi(v(\cdot,t_j))$ converges to $\Phi(U)$ 
in measure. 
Then, by Fatou's Lemma,
\[
 \lyap{\Phi(U)}\le\liminf_{j\to\infty} \lyap{\Phi(v(\cdot,t_j))}\,.
\]
On the other hand, by Theorem~\ref{exist-ima}, 
\[
\lyap{\Phi(v(\cdot,t))}
\le \lyap{\Phi(u_0)}	
\,,\qquad
\text{for all $t>0$.}
\]
By assumption, we have
\[
 \lyap{\Phi(u_0)} <\Lambda_{2}
\]
and \eqref{fineq} follows by pairing this strict inequality with the previous two ones.
\qed
\begin{remark}
It would be interesting to upgrade the $L^{m+1}$ convergence in \eqref{mainthm1} to the \emph{uniform} (resp., {\em local uniform}) convergence. Once $C^{\alpha}$ regularity up to the boundary (resp., interior $C^\alpha$ regularity) is available, which is still not the case for the sign changing solutions, it would be sufficient to reproduce the argument of \cite[Chapter 20, page 526]{VaBook}.
\end{remark}

\subsection{Proof of Proposition~\ref{prop:select}} 
Assume that one of the two conditions in \eqref{assumnonloc} holds and set $q=(m+1)/m$. Then, by Proposition~\ref{prop:encontr}
there exists $\theta\in C([0,1];\sob)$ such that $\theta(\cdot, 0)$ equals
 the positive minimizer  $w$ of $\lyapno$, moreover $\theta(\cdot,1)=\Phi(u_0)$ and 
\begin{equation}
\label{propcntrl1}
 \lyap{\Phi(\theta(\cdot,t))}<\Lambda_2 \,, \qquad\text{for all $t\in[0,1]$.}
\end{equation}
 Now, we set
\[
	z(\cdot,t)= \begin{cases}
		\theta(\cdot,t)\,, & \qquad \text{if $0\le t\le1$}\\
		\Phi(v(\cdot,t-1))\,,& \qquad \text{if $t>1$,}
	\end{cases}
\]
where $v$ is the unique solution of \eqref{FPMEs} with $v(0)=u_0$. Then, in view of Theorem~\ref{exist-ima}
and of Proposition~\ref{prop:encontr}, we have
\begin{equation}
\label{propctrl2}
	\lyap{z(\cdot,t)} <\Lambda_{2}
\end{equation}

Hence, by coercivity (see Lemma~\ref{lm:coerc}) we deduce that the trajectory $z(\cdot,t)$ is contained in a bounded
subset of $\sob$.
Since $t\mapsto z(\cdot,t)$ is continuous 
from $[0,1]$ to $\sob$,  by the compactess of the embedding into $L^q(\Omega)$ it is continuous as a function with values in $L^q(\Omega)$, as well; also, it is easily seen that the continuity of $t\mapsto v(\cdot,t-1)$ from $[1,+\infty)$ to $L^{m+1}(\Omega)$
implies that of $t\mapsto z(\cdot,t)= \Phi(v(\cdot,t-1))$ from $[1,+\infty)$ to $L^q(\Omega)$. Therefore,
$z$ belongs both to $L^\infty((0,+\infty);\sob)$ and to $C([0,+\infty);L^0(\Omega))$.

Now, we argue by contradiction, and we assume $v(\cdot,t)$ to converge in measure to $-w$.
Then, so does $z(\cdot,t)$ and it follows that $z$ is then eligible for the mountain pass formula
of Proposition~\ref{prop:mszeta}. Thus, 
\[
\Lambda^\ast=\inf_{z\in\mathfrak{Z}}\sup_{t\in[0,+\infty)} 
	\lyap{z(\cdot,t)}<\Lambda_2\,.
\]
in contradiction with Proposition~\ref{prop:MPexcited}.\qed

\appendix
\section{An elementary inequality}

\begin{lemma}\label{lm:A2}
We set $f(t)=\tfrac{1}{m+1}|t|^{m+1}$, $g(t)=|t|^{\frac{m-1}{2}}t$, and we let $a,b\in\mathbb R$. Then
\[
	f(a)-f(b)-f'(b)(a-b) \ge C_0(m) |g(b)-g(a)|^2\,,
\]
where $C_0(m)=(m+1)^{-3}$.
\end{lemma}
\begin{proof}
We claim that
\begin{equation}
\label{proofA.1}
\frac{1}{2}|a|^{m+1} + \frac{1}{2}|b|^{m+1}\ge \left\vert\frac{a+b}{2}\right\vert^{m+1} + \frac{1}{8}\max\big\{|a|^{m-1}\mathbin{,}|b|^{m-1}
\big\} |a-b|^2\,.
\end{equation}
Thence, since - by strict convexity - we also have
\[
\left\vert\frac{a+b}{2}\right\vert^{m+1}\ge |b|^{m+1}+(m+1)|b|^{m-1}b\ \frac{a-b}{2}\,,
\]
we would arrive at
\[
	|a|^{m+1}\ge |b|^{m+1}+ (m+1)|b|^{m-1}b\ (a-b) +\frac14\max\big\{|a|^{m-1}\mathbin{,}|b|^{m-1}
\big\} |a-b|^2\,.
\]
By Lagrange mean value theorem applied to the function $g(v)=|v|^\frac{m-1}{2}v$, we also have
\[
|g(a)-g(b)|^2 \le \frac{(m+1)^2}{4}\max\big\{|a|^{m-1}\mathbin{,}|b|^{m-1}
\big\} |a-b|^2\,,
\]
and because of the definition of $f$ we would get the conclusion by combining the last two inequalities.

Then, we are left to prove \eqref{proofA.1}. To do so, 
using Cauchy integral remainder theorem 
we write
\[
\begin{split}
f(a) & = f\left(\tfrac{a+b}{2}\right)+\tfrac12f'\left( \tfrac{a+b}{2}\right)(a-b)+\tfrac14\int_0^1f''\left(
\lambda a +(1-\lambda)\tfrac{a+b}{2}\right)(a-b)^2(1-\lambda)\,d\lambda\,,\\
f(b) & = f\left(\tfrac{a+b}{2}\right)+\tfrac12f'\left( \tfrac{a+b}{2}\right)(a-b)+\tfrac14\int_0^1f''\left(
\lambda b +(1-\lambda)\tfrac{a+b}{2}\right)(b-a)^2(1-\lambda)\,d\lambda\,.
\end{split}
\]
Since $f''(t)=m|t|^{m-1}$, if follows that
\[
\begin{split}
\frac12 f(a)+\frac12 f(b) \ge f\left(\tfrac{a+b}{2}\right) &+\frac{m}{8}(a-b)^2
\int_0^1\left\vert \lambda a + (1-\lambda)\tfrac{a+b}{2}\right\vert^{m-1}(1-\lambda)\,d\lambda\\
& + \frac{m}{8}(a-b)^2\int_0^1\left\vert \lambda b + (1-\lambda)\tfrac{a+b}{2}\right\vert^{m-1}(1-\lambda)\,d\lambda\,.
\end{split}
\]
We assume with no restriction that $|a|\ge |b|$. Hence, by the triangle inequality we see that
\[
\begin{split}
\int_0^1\left\vert \lambda a + (1-\lambda)\tfrac{a+b}{2}\right\vert^{m-1}(1-\lambda)\,d\lambda
& = \int_0^1\left\vert \tfrac{1+\lambda}{2}a +\tfrac{1-\lambda}{2}b\right\vert^{m-1}(1-\lambda)\,d\lambda\\
& \ge \int_0^1\left( \tfrac{1+\lambda}{2}a -\tfrac{1-\lambda}{2}b\right)^{m-1}(1-\lambda)\,d\lambda\\
& \ge |a|^{m-1}\int_0^1\lambda^{m-1}(1-\lambda)\,d\lambda =\frac{|a|^{m-1} }{m(m+1)}\,,
\end{split}
\]
and $|a|^{m+1}=\max\{|a|^{m-1},|b|^{m-1}\}$ by assumption. Then, by inserting the latter in the previous inequality
we get \eqref{proofA.1}, as desired.
\end{proof}

\end{document}